\documentclass[12pt]{amsart}
\usepackage{amssymb,amsmath,amscd}
\theoremstyle{plain}
\newtheorem{theorem}{Theorem}[section]

\newtheorem*{quest}{Main Question}
\newtheorem{corollary}[theorem]{Corollary}
\newtheorem{obs}[theorem]{Observation}

\theoremstyle{definition}
\newtheorem{definition}{Definition}
\numberwithin{equation}{section}
\newcommand{\D}{\displaystyle}
\renewcommand{\div}{\operatorname{div_{\mathcal{H}}}}

\newcommand{\eucl}{\operatorname{eucl}}

\newcommand{\ipg}[2]{\langle #1,#2 \rangle_\mathbb{G}}
\newcommand{\iph}[2]{\langle #1,#2 \rangle_\mathbb{H}}
\newcommand{\tp}{\texttt{p}}
\renewcommand{\span}{\operatorname{span}}
\newcommand{\divergence}{\operatorname{div}}
\begin{document}
\title[Generalizations of a Laplacian-Type Equation]{Generalizations of a Laplacian-Type Equation in the Heisenberg Group and a Class of Grushin-Type Spaces}
\author{Thomas Bieske}
\author{Kristen Childers}
\address{Department of Mathematics\\
University of South Florida\\ 
Tampa, FL 33620, USA}
\email{tbieske@math.usf.edu}
\address{Department of Mathematics\\
University of South Florida\\ 
Tampa, FL 33620, USA}
\email{childers@usf.edu}
\date{December 17,2011}

\begin{abstract}
In \cite{BGG}, Beals, Gaveau and Greiner find the fundamental solution to a $2$-Laplace-type equation in a class of sub-Riemannian spaces. This solution is related to the well-known fundamental solution to the $\tp$-Laplace equation in Grushin-type spaces \cite{BG} and the Heisenberg group \cite{CDG}. We extend the $2$-Laplace-type equation to a $\tp$-Laplace-type equation. We show that the obvious generalization does not have  desired properties, but rather, our generalization preserves some natural properties. 
\end{abstract}
\maketitle
\section{Introduction and Motivation}
In \cite{BGG}, fundamental solutions to a generalization of the $2$-Laplace equation were found in a wide class of sub-Riemannian spaces.  This class includes some of the spaces in \cite{CDG, HH, BG}. The methodology of \cite{BGG} mixes the geometric properties of the space with the linearity of the 2-Laplace operator.  In this article, we study the generalization of \cite{BGG} and look to extend it to an equation based on the $\tp$-Laplace equation for $1<\tp<\infty$. Because the $\tp$-Laplace equation is nonlinear, we face some technical issues, the first of which is the proper way to generalize the original equation. In Section 3 we discuss the original equations of \cite{BGG} and in Section 4, we find that a seemingly ``natural" generalization is not optimal. In Section 5, we will find a generalization that extends the $2$-Laplace equation while maintaining the connection to the fundamental solutions of the $\tp$-Laplace equation and in Section 6, we explore the limiting case as $\tp \to \infty$. We focus on two specific classes of sub-Riemannian spaces, namely, Grushin-type planes, which are two-dimensional sub-Riemannian spaces lacking a group law, and the Heisenberg group, a sub-Riemannian space possessing a group law. For the sake of completeness, we highlight the key properties of our environments in Section 2.  

This article is based on the Master's Thesis of the second author at the University of South Florida under the direction of the first author. The second author wishes to thank the Department of Mathematics and Statistics at the University of South Florida for the research opportunity that led to both the thesis and this article. 
\section{The Environments}
We concern ourselves with two sub-Riemannian environments, the Heisenberg group and Grushin-type planes, which are 2-dimensional Grushin-type spaces. We will recall the construction of these spaces and then highlight their main properties.
\subsection{The Heisenberg Group}
We begin with $\mathbb{R}^{2n+1}$ using the coordinates $(x_1,x_2,\ldots, x_{2n},z)$ and consider the linearly
independent vector fields $ \{X_i,Z\}, $ where the index $i$ ranges from $1$ to $2n$, defined by  
\begin{eqnarray*}
 	X_i & = & \begin{cases} 
  			\frac{ \D \partial }{\D \partial x_i} - 
			\frac{ \D x_{n+i}} { \D 2}
			 \frac{ \D \partial }{\D \partial z}& 
			            \text{ if  $1 \leq i\leq n$} \\ \\
			\frac{ \D \partial }{ \D \partial x_{i}} + 
			\frac{\D x_{i-n}} { \D 2}
			 \frac{ \D \partial }{\D \partial z} 
			         & \text{ if  $n < i \leq 2n$,}
			\end{cases}\\
 &   &  \\
Z & = & \frac{ \D \partial }{ \D \partial z}  \, \cdot
\end{eqnarray*} 
For $i \leq j$, these vector fields obey the relations 
\begin{equation*}
 	[X_i,X_j] =  \begin{cases} 
  			Z & \text{ if  $j=i+n$} \\ 
			0 & \text{ otherwise}
			\end{cases}
\end{equation*}
and for all $i$, 
\begin{equation*}
[X_i,Z] = 0.
\end{equation*}
We then have a Lie Algebra denoted $h_n$ that decomposes as a direct sum 
\begin{equation*}
h_n = V_1 \oplus V_2 
\end{equation*}
where $V_1$ is
spanned by the $X_i$'s and $V_2$ is spanned by $Z$. 
We endow $h_n$  with an inner
product $\iph{\cdot \,}{\cdot}$ and related norm $\|\cdot\|_{\mathbb{H}}$ so that this basis is
orthonormal. 
The corresponding Lie Group is called the
general Heisenberg group of dimension $n$ and is denoted by $\mathbb{H}^n$. With this choice of vector fields the exponential map can be used to identify elements of $h_n$ and $\mathbb{H}^n$ with each other via
$$\sum_{i=1}^{2n}x_iX_i+ zZ \in h_n \leftrightarrow(x_1, x_2, \ldots, x_{2n},z)\in \mathbb{H}^n.$$
In particular,  for any $p,q$ in $\mathbb{H}^n$, written as 
$p=(x_1, x_2, \ldots, x_{2n},z_1)$ 
and $q=(\hat{x}_1, \hat{x}_2, \ldots, \hat{x}_{2n},z_2)$ the group multiplication law is given by 
\begin{eqnarray*}
p \cdot q = \mathbf{(} x_1+\hat{x}_1, x_2+\hat{x}_2, \ldots, x_{2n}+\hat{x}_{2n}, z_{1}+z_{2} + \frac{1}{2}\sum_{i=1}^{n}
(x_i\hat{x}_{n+i}-x_{n+i}\hat{x}_i) \mathbf{)}.
\end{eqnarray*}

The natural metric on $\mathbb{H}^n$ is the Carnot-Carath\'{e}odory metric  given by 
$$d_C(p,q)= \inf_{\Gamma} \int_{0}^{1} \| \gamma '(t) \|_{\mathbb{H}} dt $$
where the set $ \Gamma $
is the set of all curves $ \gamma $ such that $\gamma (0) = p, \gamma (1) = q$ and $\gamma'(t) \in V_1$.  By Chow's theorem (See, for example,
\cite{BR:SRG}.) any two points
 can be connected by such a curve, which
makes $d_C(p,q)$ a left-invariant metric on $\mathbb{H}^n$.
This metric induces a
homogeneous norm on $\mathbb{H}^n$, denoted $|\cdot|$, by $$|p|=d_C(0,p)$$ and we have the estimate $$|p| \sim \sum_{i=1}^{2n}|x_i|+|z|^{\frac{1}{2}}.$$
%This estimate leads us to define the left-invariant gauge $\mathcal{N}$ that is comparable to the  Carnot-Carath\'{e}odory metric and is given by  $$\mathcal{N}(p)=\bigg((\sum_{i=1}^{2n}x_i^2)^2+16z^2\bigg)^{\frac{1}{4}}_.$$ We define the Carnot-Carath\'{e}odory balls $B(p,r)$ and the gauge balls $B_\mathcal{N}(p,r)$ in the obvious way.

Given a smooth function $u:\mathbb{H}^n \to \mathbb{R}$, we define the horizontal gradient by
$$\nabla_0 u = (X_1u,X_2u,\ldots, X_{2n}u),$$ 
the full gradient by $$\nabla u = (X_1u,X_2u,\ldots, X_{2n}u,Zu),$$
and the symmetrized horizontal second
derivative matrix $(D^2u)^{\star}$  by 
\begin{eqnarray*} 
((D^2u)^{\star})_{ij}=\frac{1}{2}(X_iX_ju + X_jX_iu).
\end{eqnarray*}
Additionally, given a vector field $F=\sum_{i=1}^{2n}f_iX_i+f_{2n+1}Z$, we define the Heisenberg divergence of $F$, denoted $\div F$, 
by $$\div F = \sum_{i=1}^{2n}X_if_i.$$ A quick calculation shows that when $f_{2n+1}=0$, we have $$\div F = \textmd{div}_{\eucl}F$$ where 
$\textmd{div}_{\eucl}$ is the standard Euclidean divergence. 
The main operator we are concerned with is the 
horizontal $\tp$-Laplacian for
$1< \tp <\infty$ defined by \begin{eqnarray}
\Delta_\tp u & = & \div(\|\nabla_{0} u\|_{\mathbb{H}}^{\tp-2}\nabla_{0}u)  =  \sum_{i=1}^{2n}X_i\big(\|\nabla_{0} u\|_{\mathbb{H}}^{\tp-2}X_iu\big) \nonumber\\
& = & \sum_{i=1}^{2n}\frac{1}{2}(\tp-2)\|\nabla_{0} u\|_{\mathbb{H}}^{\tp-4}X_i\|\nabla_{0} u\|_{\mathbb{H}}^{2}X_iu+ \sum_{i=1}^{2n}\|\nabla_{0} u\|_{\mathbb{H}}^{\tp-2}X_iX_iu   \label{reductionh}\\
%&  & \mbox{} +\frac{1}{2}(\tp-2)\|\nabla_{0} u\|_{\mathbb{H}}^{\tp-4}X_2\|\nabla_{0} u\|^{2}X_2u+\|\nabla_{0} u\|^{\tp-2}X_2X_2u  \nonumber\\
%& = & \frac{1}{2}(\tp-2)\|\nabla_{0} u\|_{\mathbb{H}}^{\tp-4}\big(X_1\|\nabla_{0} u\|_{\mathbb{H}}^{2}X_1u+X_2\|\nabla_{0} u\|_{\mathbb{H}}^{2}X_2u\big)  \label{reductionh} \\
%& & \mbox{} +\|\nabla_{0} u\|_{\mathbb{H}}^{\tp-2}\big(X_1X_1u+X_2X_2u\big)  \nonumber\\
& = &\|\nabla_{0} u\|_{\mathbb{H}}^{\tp-4}\bigg((\tp-2)\frac{1}{2}
\iph{\nabla_{0}\|\nabla_{0} u\|_{\mathbb{H}}^{2}}{\nabla_{0} u}
 +\|\nabla_{0} u\|_{\mathbb{H}}^{2}\sum_{i=1}^{2n}X_iX_iu\bigg).   \nonumber 
\end{eqnarray}
\begin{definition}
 A function $u: \mathbb{H}^n \to \mathbb{R}$ is $C^1_{\mathbb{H}}$ at the point $p$ if $X_iu(p)$ is continuous at $p$ for all $i=1, 2,\ldots, n$ and $u$ is $C^2_{\mathbb{H}}$ at $p$ if $X_iX_ju(p)$ is continuous at $p$ for all $i,j=1,2,\ldots, n$.
\end{definition}
 % Using the horizontal gradient, we may also define the Sobolev spaces $\wq$, $\wql$, etc. in the obvious way. 
For a more complete treatment of the Heisenberg group, the interested
reader is directed to \cite{BR:SRG}, \cite{B:HG}, \cite{F:SE}, \cite{FS:HSHG}
\cite{G:MS}, \cite{H:CCG}, \cite{K:LGHT}, \cite{St:HA} and the references therein.

\subsection{Grushin-type planes}
The Grushin-type planes differ from the Heisenberg group in that Grushin-type planes lack an algebraic group law. We begin with $\mathbb{R}^{2}$, possessing coordinates $(y_1, y_2)$, $a\in \mathbb{R}$, $c\in \mathbb{R}\setminus\{0\}$ and $n\in \mathbb{N}$.  We use them to make the vector fields:
\begin{equation*}
Y_1  =  \frac{\partial}{\partial y_1} \ \textmd{and}\ \
Y_2  =   c(y_1-a)^n\frac{\partial}{\partial y_2}.
\end{equation*}
For these vector fields, the only (possibly) nonzero Lie bracket is
\begin{equation*}
[Y_1,Y_2]=Y_1Y_2-Y_2Y_1=cn(y_1-a)^{n-1}\frac{\partial}{\partial y_2}.
\end{equation*}
Because $n$ is a natural number, we see that applying the Lie bracket $n$ number of times gives us a nonzero vector at $y_1=a$, namely,
\begin{equation*}
Z\equiv[Y_1,[Y_1,\cdots[Y_1,Y_2]]\cdots]= cn!\frac{\partial}{\partial y_2}.
\end{equation*}
Since $Y_1$ and $Y_2$, as defined above, span $\mathbb{R}^{2}$ when $y_1 \neq a$, and $Y_1$ and $Z$ span $\mathbb{R}^2$ when $y_1=a$, it follows that H\"{o}rmander's condition is satisfied by these vector fields.

We will put a (singular) inner product on $\mathbb{R}^{2}$, denoted $\ipg{\cdot}{\cdot}$,  with related norm $\|\cdot\|_{\mathbb{G}}$, so that the collection $\{Y_{1}, Y_{2}\}$ 
forms an orthonormal basis. We then have a sub-Riemannian space that we will call $g_{n}$, which is also the tangent space to a generalized Grushin-type plane $\mathbb{G}_n$. Points in $\mathbb{G}_n$ will also be denoted by
$p=(y_{1}, y_{2})$.  The Carnot-Carath\'{e}odory distance on $\mathbb{G}_n$ is defined for points $p$ and $q$ as follows:
\begin{eqnarray*}
d_{\mathbb{G}}(p,q)=\inf_{\Gamma}\int \|\gamma'(t)\|_{\mathbb{G}}\;dt.
\end{eqnarray*} 
Here $\Gamma$ is the set of all curves $\gamma$ such that $\gamma(0)=p$, $\gamma(1)=q$ and
\begin{equation*}
\gamma'(t)\in \span\{Y_{1}(\gamma(t)),Y_{2}(\gamma(t))\}.
\end{equation*}
By Chow's theorem, this is an honest metric.  The Carnot-Carath\'{e}odory distance between  $p=(y_1,y_2)$ and $q=(\widehat{y_1},\widehat{y_2})$, can be estimated via Theorem 7.34 in \cite{BR:SRG} by  
\begin{equation*}
d_\mathbb{G}(p,q) \approx |\widehat{y_1}-y_1|+|\widehat{y_2}-y_2|^{\frac{1}{n+1}}.  
\end{equation*}
We shall now discuss calculus on the Grushin-type planes. Given a smooth function $f$ on $\mathbb{G}_n$, we define the horizontal gradient of $f$ as
\begin{equation*}
 \nabla_{0} f(p) = \big(Y_1f(p),Y_2f(p)\big).
\end{equation*}
We also consider the symmetrized $2\times 2$ second-order horizontal derivative matrix with entries given by
\begin{equation*}
((D^2f(p))^\star)_{ij} \frac{1}{2}\big(Y_iY_jf(p)+Y_jY_if(p)\big)
\end{equation*}

Similar to the Heisenberg case, we have the following natural definitions: 
\begin{definition} 
\mbox{}\\
\begin{itemize}
\item The function $f : \mathbb{G}_n \to \mathbb{R}$ is said to be $C^1_{\mathbb{G}}$ at the point $p=(y_1,y_2)$ with $y_1 \neq a$ if $Y_if$ is continuous at $p$ for $i = 1, 2$. Similarly, the function $f$ is $C^2_{\mathbb{G}}$ at $p$ if $Y_iY_jf$ is continuous at $p$ for $i,j=1,2$.
\item The function $f : \mathbb{G}_n \to \mathbb{R}$ is said to be $C^1_{\mathbb{G}}$ at the point $p=(a,y_2)$ if $Y_1f$ is continuous at $p$. Similarly, the function $f$ is $C^2_{\mathbb{G}}$ at $p$ if $Y_1Y_1f$ is continuous at $p$ and, if $n =1$, $Y_1Y_2f$ is continuous at $p$.
\end{itemize}
\end{definition}

Using these derivatives, we consider a key operator on $C^2_{\mathbb{G}}$ functions, namely the \tp-Laplacian for $1<\tp<\infty$, given by
\begin{eqnarray}
\Delta_\tp f & = & \divergence(\|\nabla_{0} f\|_{\mathbb{G}}^{\tp-2}\nabla_{0}f)   =  Y_1\big(\|\nabla_{0} f\|_{\mathbb{G}}^{\tp-2}Y_1f\big)+Y_2\big(\|\nabla_{0} f\|_{\mathbb{G}}^{\tp-2}Y_2f\big)  \nonumber\\
& = & \frac{1}{2}(\tp-2)\|\nabla_{0} f\|_{\mathbb{G}}^{\tp-4}Y_1\|\nabla_{0} f\|_{\mathbb{G}}^{2}Y_1f+\|\nabla_{0} f\|_{\mathbb{G}}^{p-2}Y_1Y_1f  \nonumber\\
&  & \mbox{} +\frac{1}{2}(\tp-2)\|\nabla_{0} f\|_{\mathbb{G}}^{\tp-4}Y_2\|\nabla_{0} f\|_{\mathbb{G}}^{2}Y_2f+\|\nabla_{0} f\|_{\mathbb{G}}^{p-2}Y_2Y_2f \nonumber\\
& = & \frac{1}{2}(\tp-2)\|\nabla_{0} f\|_{\mathbb{G}}^{\tp-4}\big(Y_1\|\nabla_{0} f\|_{\mathbb{G}}^{2}Y_1f+Y_2\|\nabla_{0} f\|_{\mathbb{G}}^{2}Y_2f\big)  \label{reductiong}\\
& & \mbox{} +\|\nabla_{0} f\|_{\mathbb{G}}^{\tp-2}\big(Y_1Y_1f+Y_2Y_2f\big)  \nonumber\\
& = &\|\nabla_{0} f\|_{\mathbb{G}}^{\tp-4}\bigg((\tp-2)\frac{1}{2}
\ipg{\nabla_{0}\|\nabla_{0} f\|_{\mathbb{G}}^{2}}{\nabla_{0} f}+\|\nabla_{0} f\|_{\mathbb{G}}^{2}\big(Y_1Y_1f+Y_2Y_2f\big)\bigg). \nonumber  
\end{eqnarray}
\section{Motivating Results}
\subsection{Grushin-type Planes}
Bieske and Gong \cite{BG} proved the following in the Grushin-type planes.
\begin{theorem}[\cite{BG}]\label{T6}
Let $1<\tp<\infty$ and define $$f(y_1,y_2)=c^2(y_1-a)^{(2n+2)}+(n+1)^2 (y_2-b)^2.$$  
For $\tp \neq n+2$, consider $$\tau_{\tp}=\frac{n+2-\tp}{(2n+2)(1-\tp)}$$
so that in $\mathbb{G}_n\setminus\{(a,b)\}$  we have the well-defined function
\begin{eqnarray*}
\psi_{\tp}=\left\{\begin{array}{cc}
f(y_1,y_2)^{\tau_{\tp}} & \tp \neq n+2 \\
\log f(y_1,y_2) & \tp = n+2.
\end{array}\right.
\end{eqnarray*}
Then, $\Delta_{\tp}\psi_{\tp}=0$ in $\mathbb{G}_n\setminus\{(a,b)\}$.
\end{theorem} 
In the Grushin-type planes, Beals, Gaveau and Greiner \cite{BGG} extend this equation as shown in the following theorem. 
\begin{theorem}[BGG]\label{MG}
Let $L \in \mathbb{R}$. Consider the following quantities,
\begin{equation*}
\alpha = \frac{-n}{(2n+2)}(1+L) \ \ \textmd{and}\ \ \beta  =  \frac{-n}{(2n+2)}(1-L)
\end{equation*}
where $L\in\mathbb{R}$.
We use these constants with the functions
\begin{eqnarray*}
g(y_1,y_2) & = & c(y_1-a)^{n+1}+i(n+1)(y_2-b)\\ 
h(y_1,y_2) & = & c(y_1-a)^{n+1}-i(n+1)(y_2-b)
\end{eqnarray*}
to define our main function $f(y_1,y_2)$, given by
\begin{eqnarray*}
f(y_1,y_2) & = & g(y_1,y_2)^{\alpha}h(y_1,y_2)^{\beta}
\end{eqnarray*}
Then, $\Delta_{2}f+iL[Y_1,Y_2]f=0$ in $\mathbb{G}_n\setminus\{(a,b)\}$ .
\end{theorem}

\subsection{The Heisenberg Group}
Capogna, Danielli, and Garofalo \cite{CDG} proved the following theorem.
\begin{theorem}[\cite{CDG}]\label{T8}
Let $1<\tp<\infty$. In $\mathbb{H}^n \setminus {\{0}\}$, let $$u(x_1,x_2,\ldots, x_{2n},z)=\Bigg(\bigg(\sum_{i=1}^{2n}x_i^2\bigg)^2+16z^2\Bigg).$$ For $\tp \neq 2n+2$, let $$\eta_{\tp}=\frac{2n+2-\tp}{4(1-\tp)},$$ and let
\begin{eqnarray*}
\zeta_{\tp}=\left\{\begin{array}{cc}
u(x_1,x_2,\ldots, x_{2n},z)^{\eta_{\tp}} & \tp \neq 2n+2 \\
\log u(x_1,x_2,\ldots, x_{2n},z) & \tp = 2n+2.
\end{array}\right.
\end{eqnarray*}
Then we have $\Delta_{\tp}\zeta_{\tp}=0$ in $\mathbb{H}^n \setminus {\{0}\}$.
\end{theorem}

In the Heisenberg Group, Beals, Gaveau, and Greiner \cite{BGG} extend this equation as shown in the following theorem. 

\begin{theorem}[BGG]\label{MH}
Let $L \in \mathbb{R}$. Consider the following constants,
\begin{equation*}
\eta = \frac{L-1}{2} \ \ \textmd{and}\ \ \tau  =  \frac{-(L+1)}{2} 
\end{equation*}
together with the functions,
\begin{eqnarray*}
v(x_1,x_2,\ldots, x_{2n},z) & = & \big(\sum_{j=1}^{2n} x_j^2\big) - 4 i z \\
w(x_1,x_2,\ldots, x_{2n},z) & = & \big(\sum_{j=1}^{2n}x_j^2\big) + 4 i z
\end{eqnarray*}
to define our main function, $u(x_1,x_2,\ldots, x_{2n},z)$ given by
$$u(x_1,x_2,\ldots, x_{2n},z)= v(x_1,x_2,\ldots, x_{2n},z)^\eta w(x_1,x_2,\ldots, x_{2n},z)^\tau.$$ 
Then, in $\mathbb{H}^n\setminus\{0\}$, we have $\Delta_{2}u+iL\sum_{j=1}^{n}[X_j,X_{j+n}]u=0$.
\end{theorem}
\begin{obs}
In $\mathbb{G}_n\setminus\{(a,b)\}$, we have when $\tp=2$, $$f_2(y_1,y_2)=\big(c^2(y_1-a)^{2n+2}+(n+1)^2(y_2-b)^2\big)^{\displaystyle{-\frac{n}{2n+2}}}$$ solves $$\Delta_2f_2=0.$$
Also, $$\widehat{f_L}(y_1,y_2)=g(y_1,y_2)^{-\frac{n}{2n+2}(1+L)}h(y_1,y_2)^{-\frac{n}{2n+2}(1-L)}$$ where $$g(y_1,y_2)=c(y_1-a)^{n+1}+i(n+1)(y_2-b)$$ and $$h(y_1,y_2)=c(y_1-a)^{n+1}-i(n+1)(y_2-b),$$
solves $$\Delta_2\widehat{f_L}+iL[Y_1,Y_2]\widehat{f_L}=0.$$
Notice that the equations and solutions coincide when $L=0$. That is, $$\widehat{f_0}=f_2.$$ 

Similarly, in $\mathbb{H}^n\setminus\{0\}$, $$u_2((x_1,x_2,\ldots, x_{2n},z)= \Big(\big(\sum_{i=1}^{2n}x_i^2\big)^2+16z^2\Big)^{-\frac{2n}{4}}$$ solves $$\Delta_2u_2=0.$$
Also, $$\widehat{u_L}(x_1,x_2,\ldots, x_{2n},z)=v(x_1,x_2,\ldots, x_{2n},z)^{\frac{L-1}{2}}w(x_1,x_2,\ldots, x_{2n},z)^{-\frac{(L+1)}{2}}$$ where $$v(x_1,x_2,\ldots, x_{2n},z)=\big(\sum_{j=1}^{2n}x_j^2\big)-4iz$$ and $$w(x_1,x_2,\ldots, x_{2n},z)=\big(\sum_{j=1}^{2n}x_j^2\big)+4iz$$ solves $$\Delta_2\widehat{u_L}+iL\sum_{j=1}^n[X_j,X_{j+n}]\widehat{u_L}=0.$$
Again, the equations and solutions coincide when $L=0$. So again, $$\widehat{u_0}=u_2.$$
\end{obs}
We then ask the following:
\begin{quest}
Can we extend this relationship in both $\mathbb{G}_n\setminus\{(a,b)\}$ and in $\mathbb{H}^n\setminus\{0\}$ for all $\tp,\ 1<\tp < \infty$?  That is, can we find an operator $\Phi(\tp,L)$ so that $\phi_{\tp,L}$ is a solution to $\Phi(\tp,L)\phi_{\tp,L}=0,$ for all $\tp,\ 1<\tp < \infty$ and for all $L \in \mathbb{R}$.  In addition, we should have 
\begin{equation}
\Phi(\tp,0)  = \Delta_{\tp}  \label{condition2}
\end{equation}
and
\begin{eqnarray} \begin{cases}
\Phi(2,L) = \Delta_2+iL[Y_1,Y_2] & \textmd{in the Grushin case,}\\ 
\Phi(2,L)= \Delta_2+iL\sum_{j=1}^{n}[X_j,X_{j+n}] & \textmd{in the Heisenberg case.} 
\end{cases}\label{condition1}
\end{eqnarray}
Additionally, we would like to have $\phi_{2,L}$ is the solution from \cite{BGG} and $\phi_{\tp, 0}$ is the solution from \cite{BG} in the Grushin case or \cite{CDG} in the Heisenberg case.  
\end{quest}

In order to answer this question, we first look at a good candidate for what the solution should be in each environment.
\subsection{The Core Grushin Function}
For the Grushin-type planes, we consider the following for $\tp\neq n+2$:
\begin{eqnarray*}
\alpha & = & \frac{n+2-\tp}{(1-\tp)(2n+2)}(1+L)\\
\beta & = & \frac{n+2-\tp}{(1-\tp)(2n+2)}(1-L)
\end{eqnarray*}
where $L\in\mathbb{R}$.
We use these constants with the functions
\begin{eqnarray*}
g(y_1,y_2) & = & c(y_1-a)^{n+1}+i(n+1)(y_2-b)\\ 
h(y_1,y_2) & = & c(y_1-a)^{n+1}-i(n+1)(y_2-b)
\end{eqnarray*}
to define our main function $f_{\tp,L}(y_1,y_2)$, given by
\begin{eqnarray*}
f_{\tp,L}(y_1,y_2) =\left\{\begin{array}{cc}
g(y_1,y_2)^{\alpha}h(y_1,y_2)^{\beta} & \tp \neq n+2 \\
\log\big(g^{1+L}h^{1-L}\big) & \tp = n+2.
\end{array}\right.
\end{eqnarray*}

From Theorems \ref{T6} and \ref{MG}, we have $f_{\tp,L}(y_1,y_2)$ solves 
$$\Delta_{\tp}f_{\tp,L}+iL[Y_1,Y_2]f_{\tp,L}=0$$ in $\mathbb{G}_n\setminus\{(a,b)\}$ when $\tp$ is arbitrary and $L=0$ or when $\tp=2$ and for all $L$.

\subsection{The Core Heisenberg Function}
In the Heisenberg group, for $\tp\neq 2n+2$, we consider the following quantities:
\begin{eqnarray*}
\eta & = & \frac{2n+2-\tp}{4(1-\tp)}(1-L)\\
\tau & = & \frac{2n+2-\tp}{4(1-\tp)}(1+L)
\end{eqnarray*}
where $L\in\mathbb{R}$.
We use these constants with the functions
\begin{eqnarray*}
v(x_1,x_2,\ldots,x_{2n},z) & = & \big(\sum_{j=1}^{2n}x_1^2\big)-4iz\\ 
w(x_1,x_2,\ldots,x_{2n},z) & = & \big(\sum_{j=1}^{2n}x_1^2\big)+4iz
\end{eqnarray*}
to define our main function $u_{\tp,L}(x_1,x_2,\ldots,x_{2n},z)$, given by
\begin{eqnarray*}
u_{\tp,L}(x_1,x_2,\ldots,x_{2n},z) =\left\{\begin{array}{cc}
v(x_1,x_2,\ldots,x_{2n},z)^{\eta}w(x_1,x_2,\ldots,x_{2n},z)^{\tau} & \tp \neq 2n+2 \\
\log\big(v^{1-L}w^{1+L}\big) & \tp = 2n+2.
\end{array}\right.
\end{eqnarray*}

From Theorems \ref{T8} and \ref{MH}, we have $u_{\tp,L}(x_1,x_2,\ldots,x_{2n},z)$ solves 
$$\Delta_{\tp}u_{\tp,L}+iL\sum_{j=1}^{2n}[X_j,X_{j+n}]u_{\tp,L}=0$$ in $\mathbb{H}^n\setminus\{0\}$ when $\tp$ is arbitrary and $L=0$ or when $\tp=2$ and for all $L$.
\section{A Negative Result}
A ``natural" generalization of the equation $\Delta_2\phi+iL[Z_1,Z_2]\phi$ is $\Delta_{\tp}\phi+iL[Z_1,Z_2]\phi$ where $Z_i=Y_i$ in the Grushin-type planes and $Z_i=X_i$ in the Heisenberg group $\mathbb{H}^1$.  We now consider this equation in each of our environments. 
We will suppress the subscripts on the function $f$ and on $\|\cdot\|$ for the upcoming formal computations.
\begin{theorem} \label{negg}
Let $f_{\tp,L}, \alpha,\beta$ be as in the previous section.  Let $\tp \neq n+2$ and let $L \in \mathbb{R}$ with $L \neq \pm 1$.  Then in $\mathbb{G} \setminus \{(a,b)\}$
\begin{eqnarray*}
& & \Delta_\tp f_{\tp,L}+iL(\tp-1) \|\nabla_0 f_{\tp,L} \|^{\tp-2}[Y_1,Y_2]f_{\tp,L}  \\
& &\mbox{} -\|\nabla_0 f_{\tp,L} \|^{\tp-2} \frac{L^2}{L^2-1}(-4)\left(\frac{(\tp-2)(1+n\tp)}{2+n-\tp}\right)(Y_2g^{\alpha})(Y_2h^{\beta})=0.
\end{eqnarray*} 
In particular, $\Delta_\tp f_{\tp,L} +i L [Y_1,Y_2]f_{\tp,L}$ need not be zero. 
\end{theorem}
\begin{proof}
Recall from equation \eqref{reductiong}, 
\begin{eqnarray*}
\Delta_\tp f & = & \|\nabla_{0} f\|^{\tp-4}\bigg((\tp-2)\frac{1}{2}
\Big(Y_1\|\nabla_0f\|^2(Y_1f)+Y_2\|\nabla_0f\|^2(Y_2f)\Big)\bigg) \\
&& \mbox{}+\|\nabla_{0} f\|^{2}\big(Y_1Y_1f+Y_2Y_2f\big)\bigg).
\end{eqnarray*}
Thus to show 
\begin{eqnarray*}
& & \Delta_\tp f+iL(\tp-1) \|\nabla_0 f \|^{\tp-2}[Y_1,Y_2]f  \\
& &\mbox{} -\|\nabla_0 f \|^{\tp-2} \frac{L^2}{L^2-1}(-4)\left(\frac{(\tp-2)(1+n\tp)}{2+n-\tp}\right)(X_2g^{\alpha})(X_2h^{\beta})=0,
\end{eqnarray*} 
we will consider
\begin{eqnarray*}
&& \|\nabla_{0} f\|^{\tp-4}\Big((\tp-2)\frac{1}{2}\Big(Y_1\|\nabla_0f\|^2(Y_1f)+Y_2\|\nabla_0f\|^2(Y_2f)\Big)\\
&&\mbox{}+\|\nabla_{0} f\|^{2}\big(Y_1Y_1f+Y_2Y_2f\big)\Big)+iL(\tp-1) \|\nabla_0 f \|^{\tp-2}[Y_1,Y_2]f  \\
&&\mbox{}-\|\nabla_0 f \|^{\tp-2} \frac{L^2}{L^2-1}(-4)\left(\frac{(\tp-2)(1+n\tp)}{2+n-\tp}\right)(X_2g^{\alpha})(X_2h^{\beta})\\
&=&\|\nabla_{0} f\|^{\tp-4}\Bigg(\Big((\tp-2)\frac{1}{2}\Big(Y_1\|\nabla_0f\|^2(Y_1f)+Y_2\|\nabla_0f\|^2(Y_2f)\Big)\\
&&\mbox{}+\|\nabla_{0} f\|^{2}\big(Y_1Y_1f+Y_2Y_2f\big)\Big)+iL(\tp-1) \|\nabla_0 f \|^{2}[Y_1,Y_2]f  \\
&&\mbox{}-\|\nabla_0 f \|^{2} \frac{L^2}{L^2-1}(-4)\left(\frac{(\tp-2)(1+n\tp)}{2+n-\tp}\right)(X_2g^{\alpha})(X_2h^{\beta})\Bigg).
\end{eqnarray*} 
We need only show that
\begin{eqnarray*}
&&\Big((\tp-2)\frac{1}{2}\Big(Y_1\|\nabla_0f\|^2(Y_1f)+Y_2\|\nabla_0f\|^2(Y_2f)\Big)\\
&&\mbox{}+\|\nabla_{0} f\|^{2}\big(Y_1Y_1f+Y_2Y_2f\big)\Big)+iL(\tp-1) \|\nabla_0 f \|^{2}[Y_1,Y_2]f  \\
&&\mbox{}-\|\nabla_0 f \|^{2} \frac{L^2}{L^2-1}(-4)\left(\frac{(\tp-2)(1+n\tp)}{2+n-\tp}\right)(X_2g^{\alpha})(X_2h^{\beta})=0.
\end{eqnarray*} 
To show this, we will require the following quantities. We compute for $\tp \neq n+2$,
\begin{eqnarray*}
Y_1f & = & c(n+1)(y_1-a)^{n}g^{\alpha-1}h^{\beta-1}(\alpha h+\beta g)\\
\overline{Y_1 f} & = & c(n+1)(y_1-a)^{n}g^{\beta-1}h^{\alpha-1}(\alpha g+\beta h)\\
Y_2f & = & ic(n+1)(y_1-a)^{n}g^{\alpha-1}h^{\beta-1}(\alpha h-\beta g)\\
\overline{Y_2 f} & = & ic(n+1)(y_1-a)^{n}g^{\beta-1}h^{\alpha-1}(-\alpha g+\beta h)\\
\|\nabla_{0} f\|^{2} & = &  4c^2(n+1)^2(y_1-a)^{2n}g^{\alpha+\beta-1}h^{\alpha+\beta-1}(\alpha^2+\beta^2)\\
Y_1(Y_1f) & = & c(n+1)(y_1-a)^{n-1}g^{\alpha-2}h^{\beta-2}\bigg(ngh(\alpha h+\beta g) +  \\
&& \mbox{}c(n+1)(y_1-a)^{n+1}\Big((\alpha h+\beta g)\big((\alpha-1)h-(\beta-1)g)\\
&& \mbox{} +gh(\alpha+\beta)\Big)
\bigg)\\
Y_2(Y_2f) & = & -c^2(n+1)^2(y_1-a)^{2n}g^{\alpha-2}h^{\beta-2} \times\\
&& \mbox{}\Big((\alpha h+\beta g)\big((\alpha-1)h-(\beta-1)g)-gh(\alpha+\beta)\Big).\\
\end{eqnarray*}
To proceed, we shall also need the following,
\begin{eqnarray*}
Y_1\|\nabla_{0}f\|^{2} & = & 2^2c^2(n+1)^2(\alpha ^2+\beta ^2)(y_1-a)^{2n-1}g^{\alpha+\beta-2}h^{\alpha+\beta-2} \\
&& \mbox{} \times \big(ngh+c^2(n+1)(\alpha+\beta-1)(y_1-a)^{2n+2} \big) \\
\end{eqnarray*}
and
\begin{eqnarray*}
Y_2\|\nabla_{0}f\|^{2} & = & 2^2c^3(n+1)^4(\alpha ^2+\beta ^2)(y_1-a)^{3n}(y_2-b)\\
&& \mbox{} \times(\alpha+\beta-1)g^{\alpha+\beta-2}h^{\alpha+\beta-2}. \\
\end{eqnarray*}
Using the above quantities, we now compute
\begin{eqnarray*}
\lefteqn{Y_1\|\nabla_0f\|^2(Y_1f)+Y_2\|\nabla_0f\|^2(Y_2f) =} & & \\
&&\mbox{}2^2c^3(n+1)^3(\alpha^2+\beta^2)(y_1-a)^{3n-1}g^{2\alpha+\beta-3}h^{\alpha+2\beta-3}\\
&&\mbox{}\times\Big((\alpha h+\beta g)\big(ngh+c^2(n+1)(\alpha+\beta-1)(y_1-a)^{2n+2}\big)\\
&&\mbox{}+ic(n+1)^2(y_1-a)^{n+1}(y_2-b)(\alpha+\beta-1)(\alpha h -\beta g)\Big)
\end{eqnarray*}
and
\begin{eqnarray*}
\lefteqn{\|\nabla_0f\|^2(Y_1Y_1f+Y_2Y_2f) = } & & \\ &&2c^3(n+1)^3(\alpha^2+\beta^2)(y_1-a)^{3n-1}g^{2\alpha+\beta-3}h^{\alpha+2\beta-3}\\
&&\mbox{}\times\Big(ngh(\alpha h+\beta g)+4c(n+1)(y_1-a)^{n+1}gh(\alpha\beta)\Big).
\end{eqnarray*}
We can then calculate 
\begin{eqnarray}
&&(\tp-2)\frac{1}{2}\Big(Y_1\|\nabla_0f\|^2(Y_1f)+Y_2\|\nabla_0f\|^2(Y_2f)\Big) \nonumber\\
&& \mbox{}+\|\nabla_{0} f\|^{2}\big(Y_1Y_1f+Y_2Y_2f\big) \nonumber\\
&=& \mbox{} 2c^3(n+1)^3(\alpha ^2+\beta^2)(y_1-a)^{3n-1}g^{2\alpha+\beta-3}h^{\alpha+2\beta-3} \nonumber\\
&& \mbox{}\times\Big((p-1)(\alpha h+ \beta g)ngh \nonumber\\
&&\mbox{}+(p-2)c^2(n+1)(y_1-a)^{2n+2}(\alpha+\beta-1)(\alpha h+\beta g) \nonumber\\
&&\mbox{}+ic(p-2)(n+1)^2(y_1-a)^{n+1}(y_2-b)(\alpha +\beta -1)(\alpha h -\beta g)\nonumber\\
&&\mbox{}+2^2c(n+1)(y_1-a)^{n+1}\alpha \beta gh  \nonumber
\Big).
\end{eqnarray}
\vspace{2 mm}

We will need the above quantity with
\begin{eqnarray*}
iL(\tp-1) \|\nabla_0 f \|^{2}[Y_1,Y_2]f & = & -2L(p-1)c^3n(n+1)^3(\alpha^2+\beta^2) \\
&&\mbox{} \\
&&\mbox{}\times(y_1-a)^{3n-1}(\alpha h-\beta g)g^{2\alpha+\beta-2}h^{\alpha+2\beta-2}
\end{eqnarray*}
and
\begin{eqnarray*}
\lefteqn{\|\nabla_0 f \|^{2} \frac{L^2}{L^2-1}(-4)\left(\frac{(\tp-2)(1+n\tp)}{2+n-\tp}\right)(X_2g^{\alpha})(X_2h^{\beta}) = } && \\
\\
& & \mbox{} \\
& & \mbox{} -2^3c^4(n+1)^4(y_1-a)^{4n}(\alpha^2+\beta^2)\alpha \beta g^{2\alpha+\beta-2} h^{\alpha+2\beta-2}\\
& & \mbox{} \\
& & \mbox{}\times \frac{(L^2)(\tp-2)(1-n\tp)}{(L^2-1)(2+n-\tp)}.
\end{eqnarray*}
We let $\Lambda$ be defined as
\begin{eqnarray*}
\Lambda & = & \Big((\tp-2)\frac{1}{2}\Big(Y_1\|\nabla_0f\|^2(Y_1f)+Y_2\|\nabla_0f\|^2(Y_2f)\Big) \\
& & \mbox{} \\
& & \mbox{}+\|\nabla_{0} f\|^{2}\big(Y_1Y_1f+Y_2Y_2f\big)\Big)+iL(\tp-1) \|\nabla_0 f \|^{2}[Y_1,Y_2]f  \\
& & \mbox{} \\
& & \mbox{}-\|\nabla_0 f \|^{2} \frac{L^2}{L^2-1}(-4)\left(\frac{(\tp-2)(1+n\tp)}{2+n-\tp}\right)(X_2g^{\alpha})(X_2h^{\beta}).
\end{eqnarray*}
 We then compute
\begin{eqnarray*} 
 \Lambda & = & 2c^3(n+1)^3(\alpha ^2+\beta^2)(y_1-a)^{3n-1}g^{2\alpha+\beta-3}h^{\alpha+2\beta-3} \\
&& \mbox{}\times\Big((p-1)(\alpha h+ \beta g)ngh \\
&&\mbox{}+(p-2)c^2(n+1)(y_1-a)^{2n+2}(\alpha+\beta-1)(\alpha h+\beta g) \\
&&\mbox{}+ic(p-2)(n+1)^2(y_1-a)^{n+1}(y_2-b)(\alpha +\beta -1)(\alpha h -\beta g)\\
&&\mbox{}+2^2c(n+1)(y_1-a)^{n+1}\alpha \beta gh
\Big)\\
&& \mbox{}+\Big(-2L(p-1)c^3n(n+1)^3(\alpha^2+\beta^2) \\
&&\mbox{}\times(y_1-a)^{3n-1}(\alpha h-\beta g)g^{2\alpha+\beta-2}h^{\alpha+2\beta-2}\Big)\\
&& \mbox{}-\Big(-2^3c^4(n+1)^4(y_1-a)^{4n}(\alpha^2+\beta^2)\alpha \beta g^{2\alpha+\beta-2} h^{\alpha+2\beta-2}\\
&&\mbox{}\times \frac{(L^2)(\tp-2)(1-n\tp)}{(L^2-1)(2+n-\tp)}\Big)\\
&=&2c^3(n+1)^3(\alpha^2+\beta^2)(y_1-a)^{3n-1}g^{2\alpha+\beta-3}h^{\alpha+2\beta-3} \\
&& \mbox{}\Bigg(ngh(p-1)\big((\alpha h+ \beta g)-L(\alpha h- \beta g)\big)\\
&&\mbox{}2^2c(n+1)(y_1-a)^{n+1}\alpha \beta gh\left(1+ \frac{(L^2)(\tp-2)(1-n\tp)}{(L^2-1)(2+n-\tp)} \right)\\
&&\mbox{}c(p-2)(n+1)(y_1-a)^{n+1}(\alpha +\beta -1)gh(\alpha+\beta)\Bigg)\\
&=& 2c^3(n+1)^3(\alpha^2+\beta^2)(y_1-a)^{3n-1}g^{2\alpha+\beta-3}h^{\alpha+2\beta-3} \\
& & \mbox{}\times \Bigg(n(p-1)(h+g)\left(\frac{(-1+L^2)(2+n-\tp)}{2(1+n)(\tp-1)} \right)\\
&& \mbox{}+c(n+1)(y_1-a)^{n+1}\left(-\frac{n(-1+L^2)(2+n-\tp)}{(n+1)^2} \right)\Bigg)\\
&=& 2c^3(n+1)^3(\alpha^2+\beta^2)(y_1-a)^{3n-1}g^{2\alpha+\beta-3}h^{\alpha+2\beta-3} \\
& & \mbox{}\times \Bigg( c(y_1-a)^{n+1}\left(\frac{n(-1+L^2)(2+n-\tp)}{(n+1)}\right)\\
&& \mbox{}+c(y_1-a)^{n+1}\left(-\frac{n(-1+L^2)(2+n-\tp)}{(n+1)}\right)\Bigg)=0.
\end{eqnarray*} 
\end{proof}

\subsection{The Heisenberg group}

Similar to the Grushin case (Theorem \ref{negg}), in the Heisenberg group $\mathbb{H}^1$, Theorems \ref{T8} and \ref{MH} lead us to hypothesize that $u_{\tp, L}(x_1,x_2,x_3)$ should solve $$\Delta_{\tp}u_{\tp, L}+iL[X_1,X_2]u_{\tp, L}=0$$ for $\tp$ arbitrary and $L \in \mathbb{R}$

Unfortunately, we discover this is not the case. Again, we will suppress the subscripts on the function $u$ and on $\|\cdot\|$ throughout our calculations. 

\begin{theorem} \label{negh}
Let $\tp \neq 4$. Then in $\mathbb{H}^1 \setminus \{0\},$ $$\Delta_{\tp}u+iL[X_1,X_2]u$$ need not be zero.
\end{theorem}
\begin{proof}
Recall from equation \eqref{reductionh}, 
\begin{eqnarray*}
\Delta_\tp u & = & \|\nabla_{0} u\|^{\tp-4}\bigg((\tp-2)\frac{1}{2}
\Big(X_1\|\nabla_0u\|^2(X_1u)+X_2\|\nabla_0u\|^2(X_2u)\Big)\bigg) \\
&& \mbox{}+\|\nabla_{0} u\|^{2}\big(X_1X_1u+X_2X_2u\big)\bigg).
\end{eqnarray*}
Thus to compute $$\Delta_{\tp}u+iL[X_1,X_2]u,$$ we will consider 
\begin{eqnarray*}
\|\nabla_{0} u\|^{\tp-4}\bigg((\tp-2)\frac{1}{2}
\Big(X_1\|\nabla_0u\|^2(X_1u)\\
\mbox{}+X_2\|\nabla_0u\|^2(X_2u)\Big) 
+\|\nabla_{0} u\|^{2}\big(X_1X_1u+X_2X_2u\big)\bigg)+iLZu
\end{eqnarray*}
We compute:
\begin{eqnarray*}
X_1u & = & 2v^{\eta-1}w^{\tau-1}\big((\eta w+\tau v)x_1+(\eta w-\tau v)ix_2\big)\\
\overline{X_1u} & = & 2w^{\eta-1}v^{\tau-1}\big((\eta v+\tau w)x_1-(\eta v-\tau w)ix_2\big)\\
X_2u & = & 2v^{\eta-1}w^{\tau-1}\big((\eta w+\tau v)x_2-(\eta w-\tau v)ix_1\big)\\
\overline{X_2u} & = & 2w^{\eta-1}v^{\tau-1}\big((\eta v+\tau w)x_2+(\eta v-\tau w)ix_1\big)\\
\|\nabla_0u\|^2 & = & 2^3(\eta^2+\tau^2)v^{\eta+\tau-1}w^{\eta+\tau-1}(x_1^2+x_2^2)  \\
X_1X_1u & = & 2v^{\eta - 2}w^{\tau - 2} \Big(2\big((\eta w + \tau v)x_1^2 \\ 
& & \mbox{} + (\eta w- \tau v)ix_1x_2\big)\big((\eta - 1)w+ (\tau - 1)v\big) \\
& & \mbox{} -2i\big((\eta w + \tau v)x_1x_2 + (\eta w- \tau v)ix_2^2\big)\big(-(\eta - 1)w+(\tau - 1) v\big)\\
& & \mbox{} +vw\big(2(x_1^2 + x_2^2)(\tau + \eta)+vw(\eta w + \tau v)\big)\Big)\\
X_2X_2u & = & 2v^{\eta - 2} w^{\tau - 2} \Big( 2 \big((\eta w+\tau v)x_2^2 \\ 
& & \mbox{}+ (-\eta w + \tau v)ix_1x_2\big) \big((\eta -1)w+(\tau - 1)v\big)\\ 
& & \mbox{} + 2i \big((\eta w + \tau v)x_1x_2 + (-\eta w + \tau v)ix_1^2\big)\big(-(\eta - 1)w + (\tau - 1)v\big) \\
& & \mbox{} + v w\big(2(x_1^2 + x_2^2)(\eta + \tau) + (\eta w + \tau v)\big)\Big)   \\
iLZu&=&-4Lv^{\eta-1}w^{\tau-1}(-\eta w + \tau v).
\end{eqnarray*}
To proceed, we shall require the following,
\begin{eqnarray*}
X_1\|\nabla_{0}u\|^{2} & = & 2^4(\eta^2+\tau^2)v^{\eta+\tau-2}w^{\eta+\tau-2} \\
&& \mbox{} \times \bigg(x_1vw+2(\eta+\tau-1)x_1(x_1^2+x_2^2)^2\\
&& \mbox{}-8(\eta+\tau-1)x_2z(x_1^2+x_2^2)\bigg)\\
\end{eqnarray*}
and
\begin{eqnarray*}
X_2\|\nabla_{0}u\|^{2} & = & 2^4(\eta^2+\tau^2)v^{\eta+\tau-2}w^{\eta+\tau-2} \\
&& \mbox{} \times \bigg(x_2vw+2(\eta+\tau-1)x_2(x_1^2+x_2^2)^2\\
&& \mbox{}+8(\eta+\tau-1)x_1z(x_1^2+x_2^2)\bigg).
\end{eqnarray*}
Using the above, we can now compute
\begin{eqnarray*}
\lefteqn{X_1\|\nabla_0u\|^2(X_1u)+X_2\|\nabla_0u\|^2(X_2u) =} & & \\
&& \mbox{}2^5(\eta^2 + \tau^2)v^{2 \eta + \tau - 3} w^{\eta + 2 \tau -  3}\\
&& \mbox{}\times \Big((\eta w + \tau v)vw(x_1^2 +x_2^2) + \big(2(\eta w+\tau v)(\eta + \tau - 1)(x_1^2 + x_2^2)^3\big)\\
&&\mbox{} -8(\eta w- \tau v)(\eta + \tau - 1)iz(x_1^2 + x_2^2)^2\Big)
\end{eqnarray*}
and
\begin{eqnarray*}
\lefteqn{\|\nabla_0u\|^2\big(X_1X_1u+X_2X_2u\big) =} & & \\
& & \mbox{} 2^4 (\eta^2 + \tau^2) v^{2 \eta + \tau - 3} w^{\eta + 2 \tau - 3}(x_1^2 + x_2^2)\\
&&\mbox{}\Big(2vw(\eta w + \tau v)+ 4vw(\eta + \tau) (x_1^2 + x_2^2)\\
&&\mbox{}+ 2 \big((\eta - 1) w + (\tau - 1) v\big)(\eta w + \tau v) (x_1^2 + x_2^2) \\
&&\mbox{}+ 2 \big(-(\eta - 1) w + (\tau - 1) v\big) (\eta w - \tau v) (x_1^2 + x_2^2)\Big).
\end{eqnarray*}
We can now compute
\begin{eqnarray*}
\lefteqn{(\tp-2)\frac{1}{2}\Big(X_1\|\nabla_0u\|^2(X_1u)+X_2\|\nabla_0u\|^2(X_2u)\Big) +\|\nabla_{0} u\|^{2}\big(X_1X_1u+X_2X_2u\big)=} && \\
&& -\left(\frac{1}{(-1 + \tp)^4}\right)L (1 + L^2) (-4 + \tp)^3 (x_1^2 + x_2^2) g^{\frac{
  4 + L (-4 + \tp) + 5 \tp}{4 - 4 \tp}} h^{\frac{4 + 4 L + 5 \tp - L \tp} {4 - 4 \tp}}\\
&&\mbox{}   \big(L (-4 + \tp) (x_1^2 + x_2^2) + 4 i (-1 + \tp) \tp z\big).
\end{eqnarray*}
Using this, we calculate:
\begin{eqnarray*}
\Delta_{\tp}u+iLZu
= -8 L v^{\frac{1}{2} (-3 + L)} \big(L (x_1^2 + x_2^2) - 4 i z\big) w^{\frac{1}{2}(-3 - L)}. 
\end{eqnarray*}
\end{proof}

%Chapter Five
\section{A Generalization }
\subsection{Grushin-type planes}
In the previous section, we see our first instinct does not work as well as we wished. We now will put the equation given in Theorem \ref{negg} into 
divergence form. For the Grushin-type planes, we have 
\begin{equation}
\Delta_2\phi+iL[Y_1, Y_2]\phi=\divergence_G\left( \begin{array}{c}
Y_1\phi+iLY_2\phi\\
Y_2\phi-iLY_1\phi
\end{array} \right).\label{motive}
\end{equation}
Inspired by the definition of $\Delta_{\tp}$ in Equation \eqref{reductiong}  we consider 
\begin{eqnarray}
\overline{\Delta_{\tp}}\phi=\divergence_G{\Bigg( \bigg\| \begin{array}{c}
 Y_1\phi+iLY_2\phi\\
 Y_2\phi-iLY_1\phi
\end{array} \bigg\|^{\tp-2} 
 \left( \begin{array}{c}
Y_1\phi+iLY_2\phi\\
Y_2\phi-iLY_1\phi
\end{array} \right) \Bigg)}, \label{trad}
\end{eqnarray}

Using Equation \eqref{trad}, we have the following theorem. We will again suppress the subscripts on the function $f$ and on $\|\cdot\|$. 
\begin{theorem}
On $\mathbb{G}_n\setminus\{(a,b)\}$, we have
\begin{eqnarray*}
\overline{\Delta_{\tp}}f=\divergence_G{\Bigg( \bigg\| \begin{array}{c}
 Y_1f+iLY_2f\\
 Y_2f-iLY_1f
\end{array} \bigg\|^{\tp-2} 
 \left( \begin{array}{c}
Y_1f+iLY_2f\\
Y_2f-iLY_1f
\end{array} \right) \Bigg)}=0.
\end{eqnarray*}
\end{theorem}
\begin{proof}
First, we let
\begin{eqnarray*}
 \xi  & = & \left( \begin{array}{c}
 \xi _1\\
 \xi _2
\end{array} \right)
= \left( \begin{array}{c}
Y_1f+iLY_2f\\
Y_2f-iLY_1f
\end{array} \right),
\end{eqnarray*}
and then we consider the following reduction:
\begin{eqnarray*}
\overline{\Delta_\tp} f & = & \divergence(\| \xi  \|^{\tp-2} \xi ) \\  
& = & Y_1\big(\| \xi \|^{\tp-2} \xi _1\big)+Y_2\big(\| \xi  \|^{\tp-2}\ \xi _2\big)  \\
& = & \frac{1}{2}(\tp-2)\| \xi \|^{\tp-4}Y_1\| \xi  \|^{2}\xi _1+\| \xi \|^{p-2}Y_1 \xi _1 \\ 
&  & \mbox{} +\frac{1}{2}(\tp-2)\| \xi \|^{\tp-4}Y_2\| \xi  \|^{2} \xi _2+\| \xi \|^{p-2}Y_2 \xi _2\\
& = & \frac{1}{2}(\tp-2)\| \xi \|^{\tp-4}\big(Y_1\| \xi  \|^{2} \xi _1+Y_2\| \xi \|^{2} \xi _2\big)\\  
& & \mbox{} +\| \xi \|^{\tp-2}\big(Y_1 \xi _1+Y_2 \xi _2\big)\\   
& = & \| \xi \|^{\tp-4}\Bigg(\frac{1}{2}(\tp-2)\big(Y_1\| \xi  \|^{2} \xi _1+Y_2\| \xi \|^{2} \xi _2\big) \\
& & \mbox{} +\| \xi \|^{2}\big(Y_1 \xi _1+Y_2 \xi _2\big)\Bigg)_.  
\end{eqnarray*}

Thus to show $\overline{\Delta_\tp} f=0,$ we need only show that 
\begin{equation*}
\Lambda\equiv\frac{1}{2}(\tp-2)\big(Y_1\| \xi \|^{2} \xi _1+Y_2\| \xi \|^{2} \xi _2\big)+\| \xi \|^{2}\big(Y_1 \xi _1+Y_2 \xi _2\big) = 0.
\end{equation*}

\emph{Case 1:} We compute for $\tp \neq n+2$, 
\begin{eqnarray*}
Y_1f & = & c(n+1)(y_1-a)^{n}g^{\alpha-1}h^{\beta-1}(\alpha h+\beta g)\\
Y_2f & = & ic(n+1)(y_1-a)^{n}g^{\alpha-1}h^{\beta-1}(\alpha h-\beta g)\\
Y_1f+iLY_2f & = & c(n+1)(y_1-a)^{n}g^{\alpha-1}h^{\beta-1}\big(\alpha h(1-L) +\beta g(1+L)\big)  \\
\overline{Y_1f+iLY_2f} & = & c(n+1)(y_1-a)^{n}h^{\alpha-1}g^{\beta-1}\big(\alpha g(1-L) +\beta h(1+L)\big)  \\
Y_2f-iLY_1f & = &  ic(n+1)(y_1-a)^{n}g^{\alpha-1}h^{\beta-1}\big(\alpha h(1-L) -\beta g(1+L)\big) \\
\overline{Y_2f-iLY_1f} & = &  -ic(n+1)(y_1-a)^{n}h^{\alpha-1}g^{\beta-1}\big(\alpha g(1-L) -\beta h(1+L)\big) \\
 \bigg\| \begin{array}{c}
 Y_1f+iLY_2f\\
 Y_2f-iLY_1f
\end{array} \bigg\|^{2}
& = & 2c^2(n+1)^2(y_1-a)^{2n}g^{\alpha+\beta-1}h^{\alpha+\beta-1}\\
&&\mbox{} \times\big(\alpha^2(1-L)^2+\beta^2(1+L)^2 \big). \\
\end{eqnarray*}
We then calculate:
\begin{eqnarray*}
\lefteqn{Y_1(Y_1f+iLY_2f)+Y_2(Y_2f-iLY_1f) =} & & \\
& & \left(\frac{1}{(-1+\tp)^2gh}\right)\bigg(c^2 (-1 + L^2)(1 + n)(2 + n - \tp) (-2 + \tp)\\
&&\mbox{} \times( y_1-a)^{2 n} h^{\frac{(-1 + L) (2 + n - \tp)}{2 (1 + n) (-1 + \tp)}} 
g^{-\frac{(1 + L) (2 + n - \tp)}{2 (1 + n) (-1 + \tp)}}\bigg)_.
\end{eqnarray*}
We can then calculate
\begin{eqnarray*}
Y_1\left( \bigg\| \begin{array}{c}
 Y_1f+iLY_2f\\
 Y_2f-iLY_1f
\end{array} \bigg\|^{2}\right)
& = &-\left(\frac{1}{(-1 + \tp)^3 gh} \right)\bigg(2 c^2 (-1 + L^2)^2 \\
&&\mbox{}\times(1 + n) (2 + n - \tp)^2 (y_1-a)^{-1 + 2 n} h^{\frac{1}{1 + n}+\frac{1}{1 - \tp}} \\
&&\mbox{}\times\big(c^2 (y_1-a)^{2 + 2 n} + n (1 + n) (-1 + \tp) (y-b)^2\big) g^{\frac{1}{1 + n} + \frac{1}{1 - \tp}} \bigg)\\
\end{eqnarray*}
and
\begin{eqnarray*}
Y_2\left( \bigg\| \begin{array}{c}
 Y_1f+iLY_2f\\
 Y_2f-iLY_1f
\end{array} \bigg\|^{2}\right)
& = & \left(\frac{1}{(-1 + \tp)^3 gh}\right)\bigg(2 c^3 (-1 + L^2)^2\\
&&\mbox{} \times(1 + n) (2 + n - \tp)^2 (1 + n \tp) ( y_1-a)^{3 n}h^{\frac{1}{1 + n}+ \frac{1}{1 - \tp}} \\
&&\mbox{}\times(b - y_2) g^{\frac{1}{1 + n} +\frac{1}{1 - \tp}} \bigg)_. \\
\end{eqnarray*}
Using the above quantities, we compute
\begin{eqnarray}
Y_1\left( \bigg\| \begin{array}{c}
 Y_1f+iLY_2f\\
 Y_2f-iLY_1f
\end{array} \bigg\|^{2}\right)\big(Y_1f+iLY_2f\big)
+Y_2\left( \bigg\| \begin{array}{c}
 Y_1f+iLY_2f\\
 Y_2f-iLY_1f
\end{array} \bigg\|^{2}\right)\big(Y_2f-iLY_1f\big)=\nonumber \\
 \mbox{}-\left(\frac{1}{(-1 + \tp)^4 (gh)^2}\right)\bigg(2 c^4 (-1 + L^2)^3 (1 + n) (2 + n - \tp)^3 (y_1-a)^{4 n}\label{gpart1} \\
 \mbox{}\times h^{\frac{(-3 + L) (2 + n - \tp)}{2 (1 + n) (-1 + \tp)}}
g^{-\frac{(3 + L) (2 + n - \tp)}{2 (1 + n) (-1 + \tp)}}\bigg)_. \nonumber
\end{eqnarray}
and
\begin{eqnarray*}
\lefteqn{\bigg\| \begin{array}{c}
 Y_1f+iLY_2f\\
 Y_2f-iLY_1f
\end{array} \bigg\|^{2}\big(Y_1(Y_1f+iLY_2f)+Y_2(Y_2f-iLY_1f)\big)=} && \\
& &\left(\frac{1}{(-1 + \tp)^4 (gh)^2} \right)
\bigg(c^4 (-1 + L^2)^3 (1 + n) (2 + n - \tp)^3 (-2 + \tp) (y_1-a)^{4 n} \\
&&\mbox{} \times h^{\frac{(-3 +L)(2 + n - \tp)}{2 (1 + n) (-1 + \tp)}} g^{-\frac{(3 + L) (2 + n - \tp)}{2 (1 + n) (-1 + \tp)}}\bigg).
\end{eqnarray*}
We can then calculate 
\begin{eqnarray*}
\Lambda &= & \frac{1}{2}(\tp-2)\bigg(-\left(\frac{1}{(-1 + \tp)^4 (gh)^2}\right)\Big(2 c^4 (-1 + L^2)^3 (1 + n) (2 + n - \tp)^3 (y_1-a)^{4 n}\\
&&\mbox{}\times h^{\frac{(-3 + L) (2 + n - \tp)}{2 (1 + n) (-1 + \tp)}}
g^{-\frac{(3 + L) (2 + n - \tp)}{2 (1 + n) (-1 + \tp)}}\Big)\bigg)\\
&&\mbox{}+\bigg(\left(\frac{1}{(-1 + \tp)^4 (gh)^2} \right)
\Big(c^4 (-1 + L^2)^3 (1 + n) (2 + n - \tp)^3 (-2 + \tp) (y_1-a)^{4 n} \\
&&\mbox{} \times h^{\frac{(-3 +L)(2 + n - \tp)}{2 (1 + n) (-1 + \tp)}} g^{-\frac{(3 + L) (2 + n - \tp)}{2 (1 + n) (-1 + \tp)}}\Big) \bigg)=0.
\end{eqnarray*}
So we have $\overline{\Delta_{\tp}}f=0$ when $\tp \neq n+2$.

\emph{Case 2:} \label{t p=n+2}
For $\tp = n+2$, we compute: 
\begin{eqnarray}
Y_1f & = & c(n+1)(y_1-a)^n\left(\frac{1+L}{g}+\frac{1-L}{h}\right) \nonumber\\
Y_2f & = & ic(n+1)(y_1-a)^n\left(\frac{1+L}{g}-\frac{1-L}{h}\right)\nonumber\\
Y_1f+iLY_2f & = & -\left(\frac{1}{gh}\right)\big(2c^2(L^2-1)(n+1)(y_1-a)^{2n+1}\big)\nonumber\\
\overline{Y_1f+iLY_2f} & = &Y_1f+iLY_2f \label{p=n+2 1}\\
Y_2f-iLY_1f & = &  \left(\frac{1}{gh}\right)\big(2c(L^2-1)(n+1)^2(y_1-a)^{n}(y_2-b)\big) \nonumber\\
\overline{Y_2f-iLY_1f} & = &  Y_2f-iLY_1f \label{p=n+2 2}\\
 \bigg\| \begin{array}{c}
 Y_1f+iLY_2f\\
 Y_2f-iLY_1f
\end{array} \bigg\|^{2}
& = & \left(\frac{1}{gh}\right)\big(4c^2(L^2-1)^2(n+1)^2(y_1-a)^{2n} \big). \nonumber
\end{eqnarray}

We then calculate:
\begin{eqnarray*}
\lefteqn{Y_1(Y_1f+iLY_2f)+Y_2(Y_2f-iLY_1f) =} & & \\
& & -\left(\frac{1}{gh}\right)\big(2c^2(L^2-1)n(n+1)(y_1-a)^{2n} \big).
\end{eqnarray*}
We can then calculate
\begin{eqnarray*}
Y_1\left( \bigg\| \begin{array}{c}
 Y_1f+iLY_2f\\
 Y_2f-iLY_1f
\end{array} \bigg\|^{2}\right)
& = & -\left(\frac{1}{(gh)^2}\right)\bigg(8c^2(L^2-1)^2(n+1)^2\\
& & \mbox{}\times (y_1-a)^{2n-1}\big(c^2(y_1-a)^{2n+2}+n(n+1)^2(y_2-b)^2\big) \bigg)\\
Y_2\left( \bigg\| \begin{array}{c}
 Y_1f+iLY_2f\\
 Y_2f-iLY_1f
\end{array} \bigg\|^{2}\right)
& = & -\left(\frac{1}{(gh)^2}\right)\big(8c^3(L^2-1)^2(n+1)^4\\
& & \mbox{}\times(y_1-a)^{3n}(y_2-b) \big).\\
\end{eqnarray*}
Using the above quantities, we compute
\begin{eqnarray*}
Y_1\left( \bigg\| \begin{array}{c}
 Y_1f+iLY_2f\\
 Y_2f-iLY_1f
\end{array} \bigg\|^{2}\right)\big(Y_1f+iLY_2f\big)\\
+Y_2\left( \bigg\| \begin{array}{c}
 Y_1f+iLY_2f\\
 Y_2f-iLY_1f
\end{array} \bigg\|^{2}\right)\big(Y_2f-iLY_1f\big)=
\end{eqnarray*}
$$\left(\frac{1}{(gh)^2}\right)\big(16c^4(L^2-1)^3(n+1)^3(y_1-a)^{4n} \big)$$
and
\begin{eqnarray*}
\lefteqn{\bigg\| \begin{array}{c}
 Y_1f+iLY_2f\\
 Y_2f-iLY_1f
\end{array} \bigg\|^{2}\big(Y_1(Y_1f+iLY_2f)+Y_2(Y_2f-iLY_1f)\big)=} && \\
& &-\left(\frac{1}{(gh)^2}\right)\big(8c^4(L^2-1)^3n(n+1)^3(y_1-a)^{4n} \big).\\
\end{eqnarray*}
We can then calculate 
\begin{eqnarray*}
\Lambda &= & \frac{1}{2}(\tp-2)\bigg(\left(\frac{1}{(gh)^2}\right)\big(16c^4(L^2-1)^3(n+1)^3(y_1-a)^{4n} \big)\bigg)\\
& & \mbox{}+\bigg(-\left(\frac{1}{(gh)^2}\right)\big(8c^4(L^2-1)^3n(n+1)^3(y_1-a)^{4n} \big)\bigg)\\
& = & \left(\frac{1}{(gh)^2}\right)\big(8c^4(L^2-1)^3n(n+1)^3(y_1-a)^{4n} \big)\\
& & \mbox{}-\left(\frac{1}{(gh)^2}\right)\big(8c^4(L^2-1)^3n(n+1)^3(y_1-a)^{4n} \big)=0.
\end{eqnarray*}
Thus $\overline{\Delta_{\tp}}f=0 $ for $1<\tp<\infty$ and for all $L \in \mathbb{R}$.
\end{proof}
We then conclude the following corollary.
\begin{corollary}\label{gsmooth}
Let $\tp>n+2$. The function $f_{\tp,L}$, as above, is a smooth solution to the Dirichlet problem
\begin{eqnarray*}
\left\{\begin{array}{cc}
\overline{\Delta_{\tp}}f(p)=0 & p \in \mathbb{G}_n\setminus\{(a,b)\} \\
0 & p = (a,b).
\end{array}\right.
\end{eqnarray*}
\end{corollary}
\subsection{The Heisenberg Group}
We begin by considering the $2n\times 1$ vector $\Upsilon$ with components 
\begin{eqnarray*}
\Upsilon_k= \begin{cases}
X_ku+iLX_{n+k}u & \textmd{when $k=1,2,\ldots, n$} \\
X_{k}u-iLX_{k-n}u & \textmd {when $k=n+1, n+2, \ldots, 2n$}.
\end{cases}
\end{eqnarray*}

Motivated by the Grushin case, we consider the equation 
\begin{eqnarray*}
\overline{\Delta_{\tp}}u=\div (\| \Upsilon\|^{\tp-2}\Upsilon  )=0
\end{eqnarray*}
in $\mathbb{H}^n\setminus\{0\}$. 
We then have the following theorem. We will again suppress the subscripts on the function $u$ and on $\|\cdot\|$. 
\begin{theorem}
In $\mathbb{H}^n\setminus\{0\}$, we have
\begin{eqnarray*}
\overline{\Delta_{\tp}}u=\div(\| \Upsilon\|^{\tp-2} \Upsilon)=0.
\end{eqnarray*}
\end{theorem}
\begin{proof}
Following Equation \eqref{reductionh}, we have the following reduction:
\begin{eqnarray}
\overline{\Delta_\tp} u & = & \div(\|\Upsilon \|^{\tp-2}\Upsilon) = \sum_{j=1}^{2n}X_j(\|\Upsilon\|^{\tp-2}\Upsilon_j) \nonumber\\
& = & \frac{1}{2}(\tp-2)\|\Upsilon\|^{\tp-4}\sum_{j=1}^{2n}X_j\|\Upsilon \|^{2}\Upsilon_j+\sum_{j=1}^{2n}\|\Upsilon\|^{\tp-2}X_j\Upsilon_j \nonumber \\ 
& = & \|\Upsilon\|^{\tp-4}\Bigg(\frac{1}{2}(\tp-2)\sum_{j=1}^{2n}X_j\|\Upsilon \|^{2} \Upsilon_j +\|\Upsilon\|^{2}\sum_{j=1}^{2n}X_j\Upsilon_j\Bigg)_.  \label{hform}
\end{eqnarray}
Thus to show $\overline{\Delta_\tp} u = 0,$ we will show that 
\begin{equation}\label{lambda}
\frac{1}{2}(\tp-2)\sum_{j=1}^{2n}X_j\|\Upsilon \|^{2} \Upsilon_j +\|\Upsilon\|^{2}\sum_{j=1}^{2n}X_j\Upsilon_j = 0.
\end{equation}

We first consider the case $\tp \neq 2n+2$. 

For $k=1,2,\ldots, n$, we have 
\begin{equation*}
X_ku = 2v^{\eta-1}w^{\tau-1}\big((\eta w+\tau v)x_k-(-\eta w+\tau v)ix_{n+k}\big)
\end{equation*}
and for $l=n+1, n+2, \ldots, 2n$, we have
\begin{equation*}
X_lu  =  2v^{\eta-1}w^{\tau-1}\big((\eta w+\tau v)x_l+(-\eta w+\tau v)ix_{l-n}\big).
\end{equation*}
We then have $k=1,2,\ldots, n$, 
\begin{eqnarray*}
X_ku+iLX_{n+k}u & = & 2v^{\eta-1}w^{\tau-1}\bigg((\eta w+ \tau v)(x_k+iLx_{n+k})+(\eta w- \tau v)(Lx_k+ix_{n+k})\bigg)\\
\textmd{and} \ \ \ \overline{X_ku+iLX_{n+k}u} & = & 2w^{\eta-1}v^{\tau-1}\bigg((\eta v+ \tau w)(x_k-iLx_{n+k}) +(\eta v- \tau w)(Lx_k-ix_{n+k})\bigg)
\end{eqnarray*}
and for $l=n+1, n+2, \ldots, 2n$, we have
\begin{eqnarray*}
X_lu-iLX_{l-n}u & = & 2v^{\eta-1}w^{\tau-1}\bigg((\eta w+ \tau v)(x_l-iLx_{l-n})+(\eta w- \tau v)(Lx_{l}-ix_{l-n})\bigg)\\
\textmd{and} \ \ \ \overline{X_lu-iLX_{l-n}u} & = & 2w^{\eta-1}v^{\tau-1}\bigg((\eta v+ \tau w)(x_l+iLx_{l-n})+(\eta v- \tau w)(Lx_{l}+ix_{l-n})\bigg).
\end{eqnarray*}
A routine calculation produces
\begin{equation}\label{normval}
\| \Upsilon\|^{2}
= \frac{ (\tp-(2n+2))^2 }{(\tp-1)^2}( L^2-1)^2\Big(\sum_{j=1}^{2n}x_j^2 \Big)v^{\frac{2n+ \tp}{2 - 2 \tp}}w^{\frac{2n+ \tp}{2 - 2 \tp}}
\end{equation}
which yields 
\begin{equation}\label{part1}
\sum_{j=1}^{2n}X_j(\| \Upsilon\|^{2})\Upsilon_j =
(4n+2) \frac{(\tp-(2n+2))^3}{(\tp-1)^4}(L^2-1)^3\Big(\sum_{j=1}^{2n}x_j^2 \Big)^2 v^{\mathcal{M}}w^{\mathfrak{M}}.
\end{equation}
where 
\begin{eqnarray*}
\mathcal{M}=\frac{\D L(\tp-(2n+2)) + (6n-2) + 5\tp}{\D 4(1-\tp)}& \textmd{and} &
\mathfrak{M}=\frac{\D -L(\tp-(2n+2))+5\tp+(6n-2))}{\D 4(1-\tp)}.
\end{eqnarray*}
In addition, we have 
\begin{equation}\label{part2}
\sum_{j=1}^{2n}X_j\Upsilon_j = -(2n+1)\frac{(\tp-2)(p-(2n+2))}{(\tp-1)^2}(L^2-1)\Big(\sum_{j=1}^{2n}x_j^2\Big) v^{\mathcal{C}}w^{\mathfrak{C}}
\end{equation}
where 
\begin{eqnarray*}
\mathcal{C}=\frac{\D L(\tp-(2n+2)) + (2n-2) + 3\tp}{\D 4(1-\tp)}& \textmd{and} &
\mathfrak{C}=\frac{\D -L(\tp-(2n+2))+3\tp+(2n-2))}{\D 4(1-\tp)}.
\end{eqnarray*}
Combining Equations \eqref{normval} and \eqref{part2} we obtain
\begin{equation}\label{part2combine}
\|\Upsilon\|^{2}\sum_{j=1}^{2n}X_j\Upsilon_j =
-(2n+1)(p-2)\frac{ (\tp-(2n+2))^3}{(\tp-1)^4}(L^2-1)^3\Big(\sum_{j=1}^{2n}x_j^2 \Big)^2v^{\mathcal{M}}w^{\mathfrak{M}}
\end{equation}
with $\mathcal{M}$ and $\mathfrak{M}$ defined as above. Equation \eqref{lambda} then holds. 

The proof of the case $\tp = 2n+2$ is similar and omitted. 
\end{proof}

The following corollary naturally follows:
\begin{corollary}\label{hsmooth}
Let $\tp>2n+2$. The function $u_{\tp,L}$, as above, is a smooth solution to the Dirichlet problem
\begin{eqnarray*}
\left\{\begin{array}{cc}
\overline{\Delta_{\tp}}u(p)=0 & p \in \mathbb{H}^n\setminus\{0\} \\
0 & p = 0.
\end{array}\right.
\end{eqnarray*}
\end{corollary}
\section{The limit as $\tp \to \infty$}
\subsection{Grushin-type planes}
We recall that on $\mathbb{G}_n\setminus\{(a,b)\}$, we have
\begin{eqnarray*}
\overline{\Delta_{\tp}}f & = & \divergence_G(\|\xi\|^{\tp-2}\xi) \\
 & = &  \| \xi \|^{\tp-4}\Bigg(\frac{1}{2}(\tp-2)\big(Y_1\| \xi  \|^{2} \xi _1+Y_2\| \xi \|^{2} \xi _2\big) +\| \xi \|^{2}\big(Y_1 \xi _1+Y_2 \xi _2\big)\Bigg)  
\end{eqnarray*}
with $\xi$ defined by
\begin{eqnarray*}
 \xi  & = & \left( \begin{array}{c}
 \xi _1\\
 \xi _2
\end{array} \right)
= \left( \begin{array}{c}
Y_1f+iLY_2f\\
Y_2f-iLY_1f
\end{array} \right)_.
\end{eqnarray*}
Formally letting $\tp \to \infty$, we obtain 
\begin{equation*}
\overline{\Delta_{\infty}}f = (Y_1\| \xi  \|^{2} )\xi _1+(Y_2\| \xi \|^{2}) \xi _2.
\end{equation*}
Also formally letting $\tp \to \infty$, produces the function 
\begin{equation*}
f_{\infty,L}(y_1,y_2) =\
g(y_1,y_2)^{\frac{1+L}{2n+2}}h(y_1,y_2)^{\frac{1-L}{2n+2}}
\end{equation*}
where we recall the functions $g$ and $h$ are given by
\begin{eqnarray*}
g(y_1,y_2) & = & c(y_1-a)^{n+1}+i(n+1)(y_2-b)\\ 
h(y_1,y_2) & = & c(y_1-a)^{n+1}-i(n+1)(y_2-b).
\end{eqnarray*}
We have the following theorem.
\begin{theorem}
The function $f_{\infty,L}$, as above, is a smooth solution to the Dirichlet problem
\begin{eqnarray*}
\left\{\begin{array}{cc}
\overline{\Delta_{\infty}}f_{\infty,L}(p)=0 & p \in \mathbb{G}_n\setminus\{(a,b)\} \\
0 & p = (a,b).
\end{array}\right.
\end{eqnarray*}
\end{theorem}
\begin{proof}
We may prove this theorem by letting $\tp\to\infty$ in Equation \eqref{gpart1} and invoking continuity (cf. Corollary \ref{gsmooth}). For completeness, though, we compute formally. 

We first make the following definitions:
\begin{eqnarray*}
A  \equiv  \frac{1+L}{2n+2} & \textmd{and} & 
B  \equiv  \frac{1-L}{2n+2} 
\end{eqnarray*}
so that we may compute
\begin{eqnarray*}
Y_1f & = & c(n+1)(y_1-a)^{n}g^{A-1}h^{B-1}(A h+B g)\\
Y_2f & = & ic(n+1)(y_1-a)^{n}g^{A-1}h^{B-1}(A h-B g)\\
\xi_1=Y_1f+iLY_2f & = & c(n+1)(y_1-a)^{n}g^{A-1}h^{B-1}\big(A h(1-L) +B g(1+L)\big)  \\
 & = & c^2 (1-L^2)(y_1-a)^{2n+1}g^{A-1}h^{B-1} \\
\overline{Y_1f+iLY_2f} & = & c(n+1)(y_1-a)^{n}h^{A-1}g^{B-1}\big(A g(1-L) +B h(1+L)\big)\\
\xi_2=Y_2f-iLY_1f & = &  ic(n+1)(y_1-a)^{n}g^{A-1}h^{B-1}\big(A h(1-L) -B g(1+L)\big) \\
& = &c(1-L^2)(n+1)(y_1-a)^{n}(y_2-b) g^{A-1}h^{B-1}\\
\overline{Y_2f-iLY_1f} & = &  -ic(n+1)(y_1-a)^{n}h^{A-1}g^{B-1}\big(A g(1-L) -B h(1+L)\big)\\
\|\xi\|^2= \bigg\| \begin{array}{c}
 Y_1f+iLY_2f\\
 Y_2f-iLY_1f
\end{array} \bigg\|^{2}
& = & c^2(y_1-a)^{2n}g^{A+B-1}h^{A+B-1}(1-L^2)^2. 
\end{eqnarray*}
We then have 
\begin{eqnarray*}
Y_1\|\xi\|^2 & = & 2 c^2(1-L^2)^2n(n+1)^2(y_1-a)^{2n-1}(y_2-b)^2(gh)^{\frac{-1-2n}{n+1}} \\
\textmd{and}\ \ Y_2\|\xi\|^2 & = & -2 c^3(1-L^2)^2n(n+1)(y_1-a)^{3n}(y_2-b)(gh)^{\frac{-1-2n}{n+1}}
\end{eqnarray*}
so that 
\begin{eqnarray*}
Y_1\|\xi\|^2 \xi_1 & = &2 c^4(1-L^2)^3n(n+1)^2(y_1-a)^{4n}(y_2-b)^2(gh)^{\frac{-1-2n}{n+1}}g^{A-1}h^{B-1}\\
\textmd{and}\ \ Y_2\|\xi\|^2 \xi_2 & = &  -2 c^4(1-L^2)^3n(n+1)^2(y_1-a)^{4n}(y_2-b)^2(gh)^{\frac{-1-2n}{n+1}} g^{A-1}h^{B-1}.
\end{eqnarray*}
The theorem follows. 
\end{proof}
We notice that when $L=0$, this result is already well-known as Corollary 3.2 in \cite{BG}. 
In particular, combined with \cite{BG}, we have shown the following diagram commutes in $\mathbb{G}_n\setminus\{(a,b)\}$:
$$\begin{CD}
\overline{\Delta_{\tp}}f_{\tp,L}=0 @>>{\tp\to\infty}>\overline{\Delta_{\infty}}f_{\infty,L}=0 \\
@VV{L\to 0}V                @VV{L\to 0}V \\
\Delta_{\tp}f_{\tp,0}=0 @>>{\tp\to\infty}> \Delta_{\infty}f_{\infty,0}=0
\end{CD}$$
\subsection{Heisenberg group}
We recall that in $\mathbb{H}^n$, the vector $\Upsilon$ has components 
\begin{eqnarray*}
\Upsilon_k= \begin{cases}
X_ku+iLX_{n+k}u & \textmd{when $k=1,2,\ldots, n$} \\
X_{k}u-iLX_{k-n}u & \textmd {when $k=n+1, n+2, \ldots, 2n$}.
\end{cases}
\end{eqnarray*}
so that we have 
\begin{eqnarray*}
\overline{\Delta_{\tp}}u=\div (\| \Upsilon\|^{\tp-2}\Upsilon  )=0
\end{eqnarray*}
in $\mathbb{H}^n\setminus\{0\}$. 
As in the Grushin case, we formally let $\tp \to \infty$ and obtain via Equation \eqref{hform} 
\begin{equation*}
\overline{\Delta_{\infty}}u= \sum_{j=1}^{2n}(X_j\| \Upsilon  \|^{2} )\Upsilon _j.
\end{equation*}
Also formally letting $\tp \to \infty$, produces the function 
\begin{equation*}
u_{\infty,L}(x_1,x_2,\ldots,x_{2n},z) =
v(x_1,x_2,\ldots,x_{2n},z)^{\frac{1-L}{4}}w(x_1,x_2,\ldots,x_{2n},z)^{\frac{1+L}{4}}
\end{equation*}
where we recall the functions $v$ and $w$ are given by
\begin{eqnarray*}
v(x_1,x_2,\ldots,x_{2n},z) & = & \big(\sum_{j=1}^{2n}x_1^2\big)-4iz\\ 
w(x_1,x_2,\ldots,x_{2n},z) & = & \big(\sum_{j=1}^{2n}x_1^2\big)+4iz.
\end{eqnarray*}
We have the following theorem.
\begin{theorem}
The function $u_{\infty,L}$, as above, is a smooth solution to the Dirichlet problem
\begin{eqnarray*}
\left\{\begin{array}{cc}
\overline{\Delta_{\infty}}u_{\infty,L}(p)=0 & p \in \mathbb{H}^n\setminus\{(0)\} \\
0 & p = 0.
\end{array}\right.
\end{eqnarray*}
\end{theorem}
\begin{proof}
We may prove this theorem by letting $\tp\to\infty$ in Equation \eqref{part1} and invoking continuity (cf. Corollary \ref{hsmooth}). For completeness, though, we compute formally. 

We first make the following definitions:
\begin{eqnarray*}
A  \equiv  \frac{1+L}{4} & \textmd{and} & 
B  \equiv  \frac{1-L}{4} 
\end{eqnarray*}
so that we may compute
We then have $k=1,2,\ldots, n$, 
\begin{eqnarray*}
\Upsilon_k=X_ku+iLX_{n+k}u & = & 2v^{B-1}w^{A-1}\bigg((B w+ A v)(x_k+iLx_{n+k})+(B w- A v)(Lx_k+ix_{n+k})\bigg)\\
\textmd{and} \ \ \ \overline{X_ku+iLX_{n+k}u} & = & 2w^{B-1}v^{A-1}\bigg((B v+ A w)(x_k-iLx_{n+k}) +(B v- A w)(Lx_k-ix_{n+k})\bigg)
\end{eqnarray*}
and for $l=n+1, n+2, \ldots, 2n$, we have
\begin{eqnarray*}
\Upsilon_l=X_lu-iLX_{l-n}u & = & 2v^{B-1}w^{A-1}\bigg((B w+ A v)(x_l-iLx_{l-n})+(B w- A v)(Lx_{l}-ix_{l-n})\bigg)\\
\textmd{and} \ \ \ \overline{X_lu-iLX_{l-n}u} & = & 2w^{B-1}v^{\tau-1}\bigg((B v+ A w)(x_l+iLx_{l-n})+(B v- A w)(Lx_{l}+ix_{l-n})\bigg)
\end{eqnarray*}
and so 
\begin{equation*}
\| \Upsilon\|^{2}
= ( L^2-1)^2\Big(\sum_{j=1}^{2n}x_j^2 \Big)(vw)^{-\frac{1}{2}}.
\end{equation*}
For $k=1,2,\ldots, n$, 
\begin{eqnarray*}
X_k\| \Upsilon\|^{2}
 & = & ( L^2-1)^2\Big(2x_k (vw)^{-\frac{1}{2}}-2x_k\Big(\sum_{j=1}^{2n}x_j^2\Big)^2 (vw)^{-\frac{3}{2}} \Big)+8( L^2-1)^2zx_{n+k}(vw)^{-\frac{3}{2}}\sum_{j=1}^{2n}x_j^2 \\
  & = & 8z( L^2-1)^2(vw)^{-\frac{3}{2}} \Big (4z x_k +x_{n+k}\sum_{j=1}^{2n}x_j^2\Big) 
 \end{eqnarray*}
and for $l=n+1, n+2, \ldots, 2n$, we have
\begin{eqnarray*}
X_l\| \Upsilon\|^{2}
 & = & ( L^2-1)^2 \Big(2x_l (vw)^{-\frac{1}{2}}-2x_l\Big(\sum_{j=1}^{2n}x_j^2\Big)^2 (vw)^{-\frac{3}{2}} \Big)-8( L^2-1)^2zx_{l-n}(vw)^{-\frac{3}{2}}\sum_{j=1}^{2n}x_j^2 \\
  & = & 8z( L^2-1)^2(vw)^{-\frac{3}{2}}  \Big(4z x_l -x_{l-n}\sum_{j=1}^{2n}x_j^2\Big)_. 
 \end{eqnarray*}
For $k=1,2,\ldots, n$ this yields
\begin{eqnarray}
X_k\| \Upsilon\|^{2}\Upsilon_k & = & 4z( L^2-1)^2(vw)^{-\frac{3}{2}} \Big (4z x_k +x_{n+k}\sum_{j=1}^{2n}x_j^2\Big)v^{B-1}w^{A-1} \nonumber \\ 
 & & \times \bigg(((1-L) w+ (1+L) v)(x_k+iLx_{n+k})+((1-L) w- (1+L) v)(Lx_k+ix_{n+k})\bigg) \nonumber \\ 
 & = &4z( L^2-1)^2(vw)^{-\frac{3}{2}} \Big (4z x_k +x_{n+k}\sum_{j=1}^{2n}x_j^2\Big)v^{B-1}w^{A-1} \nonumber \\ 
 & & \times \Bigg((2\Big(\sum_{j=1}^{2n}x_j^2\Big)-8Liz)(x_k+iLx_{n+k}) +
(8iz-2L\sum_{j=1}^{2n}x_j^2) (Lx_k+ix_{n+k})\Bigg) \nonumber \\ 
& = &8z( L^2-1)^3(vw)^{-\frac{3}{2}} \Big (4z x_k +x_{n+k}\sum_{j=1}^{2n}x_j^2\Big)v^{B-1}w^{A-1}  (4zx_{n+k}-x_k\sum_{j=1}^{2n}x_j^2) \label{k}
\end{eqnarray}
and for $l=n+1, n+2, \ldots, 2n$, we have
\begin{eqnarray}
X_l\| \Upsilon\|^{2}\Upsilon_l
 & = & 4z( L^2-1)^2(vw)^{-\frac{3}{2}}  \Big(4z x_l -x_{l-n}\sum_{j=1}^{2n}x_j^2\Big)v^{B-1}w^{A-1}\nonumber \\ 
 & & \times
\bigg(((1-L) w+ (1+L) v)(x_l-iLx_{l-n})+((1-L) w- (1+L) v)(Lx_{l}-ix_{l-n})\bigg)\nonumber \\ 
& = &8z( L^2-1)^3(vw)^{-\frac{3}{2}} \Big (4z x_l -x_{l-n}\sum_{j=1}^{2n}x_j^2\Big)v^{B-1}w^{A-1} \nonumber \\ 
 & & \times (-4zx_{l-n}-x_l\sum_{j=1}^{2n}x_j^2) \label{l}
\end{eqnarray}
Combining Equations \eqref{k} and \eqref{l} along with an index reordering produces
\begin{equation*}
\sum_{j=1}^{2n}X_j(\| \Upsilon\|^{2})\Upsilon_j =
8z( L^2-1)^3(vw)^{-\frac{3}{2}} v^{B-1}w^{A-1}\times 0=0. 
\end{equation*}
\end{proof}
We notice that when $L=0$, this result was a part of the Ph.D. thesis of the first author \cite{BT}. 
In particular, combined with \cite{BT, B:HG}, we have shown the following diagram commutes in $\mathbb{H}^n\setminus\{0\}$:
$$\begin{CD}
\overline{\Delta_{\tp}}u_{\tp,L}=0 @>>{\tp\to\infty}>\overline{\Delta_{\infty}}u_{\infty,L}=0 \\
@VV{L\to 0}V                @VV{L\to 0}V \\
\Delta_{\tp}u_{\tp,0}=0 @>>{\tp\to\infty}> \Delta_{\infty}u_{\infty,0}=0
\end{CD}$$

\end{document}